\documentclass[12pt,english]{article}

\usepackage[T1]{fontenc}
\usepackage[utf8]{inputenc}   %
\usepackage{babel}
\usepackage{microtype}

\usepackage{amsmath}
\usepackage{amssymb}
\usepackage{amsthm}
\usepackage{dsfont}

\usepackage{textcomp}

\usepackage{newtxtext}
\usepackage[cmintegrals,varg]{newtxmath} %

\usepackage{graphicx}
\usepackage{tikz}
\usetikzlibrary{decorations.markings}
\usetikzlibrary{arrows.meta} %

\theoremstyle{plain} %
\newtheorem{theorem}{Theorem}

\newtheorem{proposition}{Proposition}

\theoremstyle{definition} %

\theoremstyle{remark} %
\newtheorem{remark}{Remark}
\newtheorem*{remark*}{Remark} %

\makeatletter
\renewenvironment{proof}[1][\proofname]{\par
  \pushQED{\qed}%
  \normalfont \topsep6\p@\@plus6\p@\relax
  \trivlist
  \item[\hskip\labelsep\bfseries#1:]\ignorespaces
}{%
  \popQED\endtrivlist\@endpefalse
}
\makeatother

\usepackage{makecell}
\usepackage{arydshln}
\usepackage{tabularx}
\usepackage{booktabs}
\newcolumntype{Y}{>{\centering\arraybackslash}X}

\usepackage{siunitx}
\usepackage{comment}
\usepackage{listings}
\usepackage{xcolor}

\definecolor{verylightgray}{gray}{0.95}
\definecolor{CommentGreen}{rgb}{0.2,0.54,0.1}

\usepackage{algorithm}
\usepackage{algpseudocode}

\usepackage[authoryear]{natbib}

\usepackage{geometry}
\usepackage{setspace}
\usepackage[colorlinks,
            linkcolor=blue,
            anchorcolor=blue,
            citecolor=blue,
            urlcolor=blue]{hyperref}
\hypersetup{pdftitle={Bartlett Couplings of the Onion and Vine LKJ Samplers},
            pdfauthor={Peter Reinhard Hansen}}

\usepackage{placeins}

\newcommand{\repourl}{\url{https://github.com/reinhardhansen/BartlettLKJ}}
\newcommand{\repodoi}{\url{https://doi.org/10.5281/zenodo.22088155}}
\newcommand{\repocommit}{\texttt{22a1764}}
\newcommand{\repotag}{\texttt{v1.2}}

\newif\ifincludetests
\includeteststrue      %

\ifincludetests
  \newcommand{\xsec}[2]{Section~\ref{#1}}
  \newcommand{\xthm}[2]{Theorem~\ref{#1}}
  \newcommand{\xalg}[2]{Algorithm~\ref{#1}}
  \newcommand{\xtab}[2]{Table~\ref{#1}}
  \newcommand{\xeq}[2]{\eqref{#1}}
  \newcommand{\testslocation}{Appendix~\ref{app:tests}}
\else
  \newcommand{\xtab}[2]{Table~#2 of the supplementary material}
  \newcommand{\testslocation}{the supplementary material}
\fi

\makeatother

\title{Bartlett Couplings of the Onion and Vine LKJ Samplers}

\author{Peter Reinhard Hansen$^{\mathsection}$
\\[0.1cm] \small $^{\mathsection}$Department of Economics, University of North Carolina at Chapel Hill,
\\ \small 107 Gardner Hall, CB 3305, Chapel Hill, NC 27599-3305, United States
\\[0.1cm] \small Corresponding author: \href{mailto:hansen@unc.edu}{\texttt{hansen@unc.edu}}}

\date{\small \today\vspace{-10mm}}

\begin{document}

\maketitle

\begin{spacing}{1.0}
\begin{abstract}
The extended-onion and C-vine constructions of \citet{LewandowskiKurowickaJoe:2009} are standard methods for sampling from the $\operatorname{LKJ}_n(\eta)$ distribution on correlation matrices. We show that both arise from the simpler row-normalized Bartlett construction associated with the restricted-Wishart representation of \citet{WangWuChu:2018}, which exposes redundancies hidden in standard gamma-based implementations of the classical constructions. Two exact row-wise couplings establish this: the squared norm of the Gaussian vector supplying the onion's direction has exactly the Gamma law required for one component of the Beta radius, and the same vector, with one chi-squared variate, generates the entire row of mutually independent C-vine partial correlations with their required symmetric-Beta laws. We also show that, relative to flat off-diagonal measure, the LKJ family maximizes entropy at fixed expected log-determinant; the dual natural parameter is $\eta-1$. Under gamma-ratio accounting, normalized Bartlett, the onion, and the conventional symmetric-Beta C-vine require $n-1$, $2(n-1)$, and $n(n-1)$ Gamma-equivalent calls. Controlled benchmarks confirm a low-dimensional advantage over the onion implementation and a persistent advantage over the C-vine implementations examined; direct Bartlett normalization also avoids subtractive complements, moving the small-$\eta$ zero-diagonal threshold from machine-epsilon scale toward the subnormal range. The sampler is valid for every real $\eta>0$ and requires only standard normal and chi-squared variates.
\end{abstract}

{\small\textit{{\noindent}Keywords:}}{\small{} LKJ distribution, correlation matrix, Bartlett decomposition, Wishart distribution, Cholesky factor, random variate generation.}{\small\par}

{\small\textit{{\noindent}MSC 2020 subject classifications:}}{\small{} primary 62E15; secondary 65C10, 62H20.}{\small\par}
\end{spacing}

\newpage

\section{Introduction}
\label{sec:introduction}

Correlation matrices are central objects in multivariate statistics, econometrics, Bayesian hierarchical modeling, and machine learning. A standard distribution on the set of $n\times n$ correlation matrices, widely used as a prior in Bayesian modeling, is the LKJ distribution of
\citet{LewandowskiKurowickaJoe:2009}, whose density on the elliptope
$$
\mathcal{E}_n=\{C\in\mathbb{R}^{n\times n}:C=C^\prime,\ C\succeq0,\ C_{ii}=1\}
$$
is proportional to
$\det(C)^{\eta-1}$ with $\eta>0$; throughout, $n\geq2$.
The special case $\eta=1$ is the uniform distribution on $\mathcal{E}_n$. Because the positive semidefinite constraint is nonlinear and difficult to impose entrywise, practical use of the LKJ distribution depends on constructive samplers that generate valid correlation matrices by design.

\citet{LewandowskiKurowickaJoe:2009} gave two constructive samplers, building on earlier work on random correlation matrices by \citet{Joe:2006}, \citet{MarsagliaOlkin:1984}, and others. The onion, or extended-onion, method builds the Cholesky factor row by row, representing each row by a random direction on a sphere and an independent beta-distributed radius; it is the computational route used by \texttt{Distributions.jl} \citep{Besancon:2021}. The C-vine method instead parametrizes the correlation matrix by $n(n-1)/2$ algebraically independent partial correlations, drawn as independent symmetric beta variates; this is the route taken by Stan \citep{Carpenter:2017,StanFunctions:2026}, whose \texttt{lkj\_corr\_cholesky\_rng} draws $n(n-1)/2$ canonical partial correlations as independent shifted symmetric-Beta variates and converts them to a Cholesky factor.\footnote{Files \texttt{stan/math/prim/prob/lkj\_corr\_cholesky\_rng.hpp} and \texttt{beta\_rng.hpp} in Stan Math v5.3.0 (Stan 2.39.0); \texttt{read\_corr\_L} is documented there as a Cholesky-factor implementation of the C-vine method. Its \texttt{beta\_rng} consumes two Gamma draws for every symmetric-Beta input: for shapes exceeding one it returns their direct ratio, while for shapes at most one it uses a log-space augmentation with two additional uniform draws. The two-Gamma count of Table~\ref{tab:rng_ledger} therefore matches Stan's Gamma-call count, although the gamma-ratio cost model is not a literal description of the small-shape branch, which is the branch taken by the last tree, where the Beta shape is $\eta$. The two classical routes are thus instantiated by the two major software ecosystems considered here.} A related parametrization by freely varying partial autocorrelations goes back to \citet{Barndorff-NielsenSchou:1973} in the stationary autoregressive setting, and \citet{JoeKurowicka:2026} have recently extended this route beyond the LKJ class, to correlation matrices with asymmetric or positive margins and targeted moments. Both constructions are exact and cost $O(n^2)$ to produce the Cholesky factor; forming $C=LL^\prime$ from it costs an additional $O(n^3)$.

There is a third route. The restricted-Wishart equivalence of \citet{WangWuChu:2018} shows that, setting $\nu=n+2\eta-1$, the correlation matrix associated with $W\sim W_n(\nu,I_n)$ has the $\operatorname{LKJ}_n(\eta)$ distribution, so drawing the Bartlett factor $A$ of $W=AA^\prime$ and normalizing each row of $A$ gives the Cholesky factor of an $\operatorname{LKJ}_n(\eta)$ matrix directly; \citet{WangWuChu:2018} also document empirically that this is faster than the extended-onion method at low and moderate dimensions. This identity is used by \citet{Hansen:CorrelationMatrix} to study the high-dimensional geometry of the elliptope. These constructions are typically presented as distinct methods rather than as coupled representations of the same underlying random quantities; the partial-correlation and Wishart routes appear as separate entries in surveys of random correlation matrix generation \citep[Sections~4.5 and~4.6]{ArchakovHansenLuo-RandomCorr:2024}, see also \citet{Pourahmadi2011}.

We show that the third route is more than a competitor of the other two: it underlies both. We establish this by two exact coupling theorems. Theorem~\ref{thm:coupling} couples the Bartlett row to the onion row: the Gaussian vector that supplies the onion's direction also carries a squared norm that is chi-squared distributed and independent of the direction; the onion discards it, and a gamma-ratio implementation then regenerates an independent copy of its law as a fresh Gamma draw, while the normalized-Bartlett sampler reuses it as one of the two Gamma components of the beta radius. Theorem~\ref{thm:vine_coupling} couples the same Bartlett row to the entire corresponding row of the C-vine: the partial correlations $\rho_{ji;1\cdots j-1}$, $j<i$, are recovered as $Z_{ij}$ divided by the norm of the remaining tail of the row, and these ratios are mutually independent with exactly the symmetric-Beta laws of \citet{LewandowskiKurowickaJoe:2009}. Thus one Gaussian vector and one chi-squared variate per row generate, simultaneously and exactly, the onion's radius--direction pair and the vine's full set of row parameters; under these couplings both classical constructions are deterministic functionals of the Bartlett factor, and each regenerates, under the gamma-ratio representation, quantities that the Bartlett route reuses during generation.

The division between known and new is the following. The restricted-Wishart equivalence, the independent symmetric-Beta laws of the C-vine parameters and their product map to the Cholesky factor, the Cholesky-factor density, and the beta--gamma independence underlying stick-breaking constructions are all known. The contributions of this paper are the explicit realization of all the independent vine inputs from the nested tails of a single Bartlett row (defined in Section~\ref{sec:vine_equivalence}), the mutual-independence proof in that representation, the resulting exact pathwise couplings of both classical samplers to the normalized Bartlett factor, their distribution-call and numerical consequences, and the maximum-entropy characterization of Section~\ref{sec:characterizations}. \citet{WangWuChu:2018} reported the low-dimensional speed advantage of the Wishart route empirically; the couplings explain its algorithmic component.

The computational implication follows from counting the scalar generator calls (normal, beta, chi-squared) that each sampler makes. Under the gamma-ratio representation of a beta variate, the ratio $X/(X+Y)$ of two independent Gammas, the onion consumes $2(n-1)$ Gamma-equivalent draws, the conventional symmetric-Beta vine $n(n-1)$, and the normalized-Bartlett sampler $n-1$. Against the onion the saving is one Gamma-equivalent draw per row and vanishes as $n\to\infty$ relative to the total cost; against the conventional vine the aggregate saving is $(n-1)^2$ Gamma-equivalent draws, approximately two per partial correlation, and it persists, with limiting ratio $2r$ in the distribution-call model, for $r$ the Gamma-to-normal cost ratio (or $1+r$ under the representation of Theorem~\ref{thm:vine_coupling}; see Section~\ref{sec:ledger}). The count is an accounting device tied to the gamma-ratio representation (Section~\ref{sec:ledger}); Section~\ref{sec:numerical} therefore separates controlled experiments that isolate the mechanisms from a package-level benchmark. The couplings also explain a numerical asymmetry: the Bartlett route computes each diagonal entry from nonnegative quantities, without the subtractions $1-R^2$ used by the textbook formulations and by the implementations examined here, so its failure threshold at small $\eta$ is the subnormal range rather than machine epsilon. The Wishart-based analysis is specific to the determinant-power family: it does not carry over to the extended vine classes of \citet{JoeKurowicka:2026}, whose densities leave that family.

The paper proceeds as follows. Section~\ref{sec:identity} states the Wishart and Bartlett identities and the normalized-Bartlett algorithm. Section~\ref{sec:cholesky_density} derives the Cholesky-factor density directly from the normalized Bartlett rows (Proposition~\ref{prop:cholesky_density}). Sections~\ref{sec:onion_equivalence} and~\ref{sec:vine_equivalence} establish the onion coupling (Theorem~\ref{thm:coupling}) and the vine coupling (Theorem~\ref{thm:vine_coupling}). Section~\ref{sec:characterizations} records three structural descriptions of the family, of which the maximum-entropy characterization appears to be new. Section~\ref{sec:ledger} gives the distribution-call ledger and cost model, the in-place implementation, and the numerical behavior of the three samplers. Section~\ref{sec:numerical} reports benchmarks; implementation tests are collected in \testslocation.

\section{The Restricted-Wishart Identity and the Bartlett Sampler}
\label{sec:identity}

The following proposition states the LKJ--restricted-Wishart equivalence of \citet{WangWuChu:2018} in the notation used in this paper; recall the elliptope $\mathcal{E}_n$ defined in Section~\ref{sec:introduction}.

\begin{proposition}[LKJ--restricted-Wishart equivalence, \citealp{WangWuChu:2018}]
\label{prop:lkj_wishart}
Let $\nu>n-1$ and let $W\sim W_n(\nu,I_n)$. Define
$$
D=\operatorname{diag}(W_{11},\ldots,W_{nn}),\qquad
C=D^{-1/2}WD^{-1/2}.
$$
Then $C$ has density proportional to
$\det(C)^{(\nu-n-1)/2}$, for $C\in\mathcal{E}_n$. Equivalently, $C\sim\operatorname{LKJ}_n(\eta)$ with
$\nu=n+2\eta-1$.
\end{proposition}

The corresponding Cholesky-factor sampler follows from the Bartlett decomposition, which holds for every real degree of freedom $\nu>n-1$ \citep{Bartlett:1939}; see \citet[Theorem~3.2.14]{Muirhead:1982} for the real-degree statement. With $\nu=n+2\eta-1$, draw a lower-triangular matrix $A$ with independent entries $A_{ij}\sim N(0,1)$ for $i>j$ and $A_{ii}^2\sim\chi^2_{\nu-i+1}$; then $W=AA^\prime$ has the $W_n(\nu,I_n)$ distribution. Standardizing $W$ to a correlation matrix is the restricted-Wishart sampler of \citet{WangWuChu:2018}; here we instead normalize the factor $A$ directly, which produces the Cholesky factor of the LKJ draw without forming the dense matrix $W$. Since $W_{ii}=\sum_{j=1}^i A_{ij}^2$, the Cholesky factor of the associated correlation matrix is obtained by normalizing each row of $A$:
$$
L_{ij}=\frac{A_{ij}}{\bigl(\sum_{k=1}^i A_{ik}^2\bigr)^{1/2}},\qquad 1\leq j\leq i,
$$
and $C=LL^\prime\sim\operatorname{LKJ}_n(\eta)$.

Equivalently, writing the sampler row by row: set $L_{11}=1$, and for $i=2,\ldots,n$ draw $Z_i\sim N_{i-1}(0,I_{i-1})$ and $G_i\sim\chi^2_{n+2\eta-i}$ independently. Setting $S_i=(Z_i^\prime Z_i+G_i)^{1/2}$, $L_{i,1:i-1}=Z_i^\prime/S_i$, and $L_{ii}=G_i^{1/2}/S_i=\{G_i/(Z_i^\prime Z_i+G_i)\}^{1/2}$, with all above-diagonal entries zero, gives
$$
L=\begin{pmatrix}
1 & 0 & \cdots & 0\\
Z_{2,1}/S_2 & G_2^{1/2}/S_2 & \cdots & 0\\
\vdots & \ddots & \ddots & \vdots\\
Z_{n,1}/S_n  & \cdots & Z_{n,n-1}/S_n & G_n^{1/2}/S_n
\end{pmatrix},
\qquad
LL^\prime\sim\operatorname{LKJ}_n(\eta).
$$

\begin{algorithm}[H]
\caption{Normalized-Bartlett $\operatorname{LKJ}_n(\eta)$ Cholesky sampler}
\label{alg:bartlett}
\begin{algorithmic}[1]
\State Set $L_{11}=1$.
\For{$i=2,\ldots,n$}
  \State Draw $Z_i\sim N_{i-1}(0,I_{i-1})$ and $G_i\sim\chi^2_{n+2\eta-i}$ independently.
  \State Set $S_i=\bigl(Z_i^\prime Z_i+G_i\bigr)^{1/2}$.
  \State Set $L_{i,1:i-1}=Z_i^\prime/S_i$ and $L_{ii}=G_i^{1/2}/S_i$.
\EndFor
\State \Return $L$ \Comment{$LL^\prime\sim\operatorname{LKJ}_n(\eta)$}
\end{algorithmic}
\end{algorithm}

This formulation shows explicitly why the sampler is valid for non-integer $\eta$: the last row requires $G_n\sim\chi^2_{2\eta}$, so the only condition is $2\eta>0$, precisely the LKJ range. Non-integer degrees of freedom are handled through $\chi^2_a\equiv\operatorname{Gamma}(a/2,2)$, where $\operatorname{Gamma}(\alpha,\theta)$ denotes the Gamma distribution with shape $\alpha$ and scale $\theta$ throughout. An in-place variant of Algorithm~\ref{alg:bartlett} is given in Section~\ref{sec:ledger}.

\begin{proof}[Proof of Proposition~\ref{prop:lkj_wishart}]
The Wishart density is proportional to $\det(W)^{(\nu-n-1)/2}\allowbreak\exp\{-\operatorname{tr}(W)/2\}$ on the positive definite cone. Write $W_{ij}=\sigma_i\sigma_j C_{ij}$ with $\sigma_i^2=W_{ii}$; then $\det(W)=\det(C)\prod_{i=1}^n\sigma_i^2$ and, since $C_{ii}=1$, $\operatorname{tr}(W)=\sum_{i=1}^n W_{ii}=\sum_{i=1}^n\sigma_i^2$, so the exponential factor depends only on $(\sigma_1,\ldots,\sigma_n)$. The absolute Jacobian determinant of the map $(\sigma_1,\ldots,\sigma_n,\{C_{ij}:i<j\})\mapsto W$ is proportional to $\prod_{i=1}^n\sigma_i^{n}$. The joint density therefore factorizes as
$$
\det(C)^{(\nu-n-1)/2}\cdot\prod_{i=1}^n\bigl[\sigma_i^{\nu-1}e^{-\sigma_i^2/2}\bigr],
$$
so the marginal density of $C$ is proportional to $\det(C)^{(\nu-n-1)/2}$ on $\mathcal{E}_n$. Setting $(\nu-n-1)/2=\eta-1$ gives $\nu=n+2\eta-1$.
\end{proof}

\section{LKJ Cholesky Factors from Normalized Bartlett Rows}
\label{sec:cholesky_density}

Most implementations of the LKJ distribution sample and store the Cholesky factor $L$ rather than the correlation matrix $C=LL^\prime$; Stan's \texttt{lkj\_corr\_cholesky} \citep{Carpenter:2017} and the \texttt{LKJCholesky} type in \texttt{Distributions.jl} \citep{Besancon:2021} are two examples. This section derives the induced density of $L$ directly from the Bartlett variables $(Z_i,G_i)$. The density itself is standard (it is the one used by Stan), but the direct normalized-Bartlett derivation exposes the independent row structure needed for the coupling in Section~\ref{sec:onion_equivalence}: it never forms $C$, and the beta-radius/uniform-direction structure that the onion construction takes as its point of departure emerges here as a consequence.

Parametrize $L$ by its strictly-lower-triangular entries: writing $x_i=L_{i,1:i-1}^\prime\in\mathbb{R}^{i-1}$, the unit row norms force $L_{ii}=(1-x_i^\prime x_i)^{1/2}$, so $L$ ranges over the product of open unit balls $\prod_{i=2}^n B^{i-1}$, where $B^{d}=\{x\in\mathbb{R}^{d}:x^\prime x<1\}$. The same ball-constrained structure, with $(1-x^\prime x)^{1/2}$ as the residual weight, arises for dynamic factor loadings in correlation models; \citet{TongHansen:2025FactorGARCH} develop a variation-free, direction-preserving Fisher parametrization of such vectors in that context.

\begin{proposition}[Cholesky-factor density]
\label{prop:cholesky_density}
Let $L$ be generated by Algorithm~\ref{alg:bartlett} with parameter $\eta>0$. Then the rows of $L$ are independent, and the joint density of $(x_2,\ldots,x_n)$ with respect to Lebesgue measure on $\prod_{i=2}^n B^{i-1}$ is
$$
p(L)=\prod_{i=2}^n
\frac{\Gamma\bigl(\eta+\frac{n-1}{2}\bigr)}{\pi^{(i-1)/2}\Gamma\bigl(\eta+\frac{n-i}{2}\bigr)}
L_{ii}^{n+2\eta-i-2}.
$$
\end{proposition}

\begin{proof}
Fix row $i$ and set $\nu_i=n+2\eta-i$ and $T_i=G_i^{1/2}$. The change of variables $g=t^2$ maps the $\chi^2_{\nu_i}$ density, proportional to $g^{\nu_i/2-1}e^{-g/2}$, to $t^{\nu_i-1}e^{-t^2/2}$, so the pair $(Z_i,T_i)$ has joint density proportional to
$$
t^{\nu_i-1}e^{-(z^\prime z+t^2)/2},
\qquad
(z,t)\in\mathbb{R}^{i-1}\times(0,\infty).
$$
Algorithm~\ref{alg:bartlett} sets the $i$th row of $L$ equal to $u=(z^\prime,t)^\prime/s$ with $s=(z^\prime z+t^2)^{1/2}$, i.e., the row is the angular part of $(Z_i,T_i)$ in the polar factorization of $\mathbb{R}^{i-1}\times(0,\infty)$, whose points we write as $su$ with $s>0$ and $u$ in the open upper half-sphere $\mathbb{S}^{i-1}_+=\{u\in \mathbb{S}^{i-1}:u_i>0\}$. Lebesgue measure factorizes as $dzdt=s^{i-1}ds\sigma(du)$, where $\sigma$ denotes surface measure on $\mathbb{S}^{i-1}$, and the joint density becomes
$$
s^{\nu_i+i-2}e^{-s^2/2}\times u_i^{\nu_i-1},
$$
where $u_i$, the last coordinate of $u$, equals $L_{ii}$. The radial and angular parts separate, so the row is independent of $S_i$ and has density proportional to $L_{ii}^{\nu_i-1}$ with respect to $\sigma$ on $\mathbb{S}^{i-1}_+$. Projecting the half-sphere onto its first $i-1$ coordinates gives $\sigma(du)=(1-x^\prime x)^{-1/2}dx=L_{ii}^{-1}dx$, hence a Lebesgue density on $B^{i-1}$ proportional to $L_{ii}^{\nu_i-2}=L_{ii}^{n+2\eta-i-2}$. The normalizing constant follows from
$$
\int_{B^{d}}(1-x^\prime x)^{b-1}dx=\frac{\pi^{d/2}\Gamma(b)}{\Gamma\bigl(b+\frac{d}{2}\bigr)},
$$
with $d=i-1$ and $b=\eta+(n-i)/2$, noting that $b+d/2=\eta+(n-1)/2$ for every row. Independence across rows holds because the pairs $(Z_i,G_i)$ occupy disjoint entries of the Bartlett factor.
\end{proof}

Four observations connect Proposition~\ref{prop:cholesky_density} to known results and implementations.
First, the density agrees with the standard LKJ Cholesky-factor density obtained by pushing $\det(C)^{\eta-1}$ through the map $C=LL^\prime$: since $\det(C)=\prod_{i=2}^n L_{ii}^2$ and the Jacobian of the change of variables from the strictly-lower-triangular entries of $L$ to those of $C$ is $\prod_{i=2}^n L_{ii}^{n-i}$, that route also gives $p(L)\propto\prod_{i=2}^n L_{ii}^{2(\eta-1)}L_{ii}^{n-i}=\prod_{i=2}^n L_{ii}^{n+2\eta-i-2}$, the density used by Stan's \texttt{lkj\_corr\_cholesky} \citep{Carpenter:2017}. The derivation above obtains it directly from Gaussian and chi-squared inputs, without forming $C$ or computing a Jacobian.

Second, rewriting the row density in radius/direction coordinates $x_i=ru$ shows that the row law factorizes: the density $\propto(1-r^2)^{\nu_i/2-1}$ against $r^{i-2}dr\sigma(du)$ separates into $R_i^2\sim\operatorname{Beta}((i-1)/2,\eta+(n-i)/2)$ and $U_i$ uniform on $\mathbb{S}^{i-2}$, independent; this is exactly the beta-radius/uniform-direction law of the extended-onion sampler. The onion construction is thus derived from the Bartlett representation, rather than merely matched to it; Theorem~\ref{thm:coupling} sharpens this into an exact algorithmic coupling.

Third, the radial factor shows $S_i^2=Q_i+G_i\sim\chi^2_{\nu_i+i-1}=\chi^2_{n+2\eta-1}=\chi^2_{\nu}$, independent of the row. The row normalization therefore discards exactly the Wishart diagonal entry $W_{ii}\sim\chi^2_\nu$, the same law for every row, which is the marginal information lost in passing from $W$ to $C$.

Fourth, the row parametrization of $L$ corresponds to the partial correlation C-vine: the entries of row $i$ are algebraic functions of the partial correlations $\{\rho_{ji;1\cdots j-1}:j<i\}$, which parametrize exactly one row each \citep[Appendix~A.4]{JoeKurowicka:2026}. The independence of rows in Proposition~\ref{prop:cholesky_density} is the distributional counterpart of the algebraic independence of the vine parameters; Theorem~\ref{thm:vine_coupling} in Section~\ref{sec:vine_equivalence} recovers the vine partial correlations and their Beta laws directly from the Gaussian and chi-squared inputs.

\section{Coupling with the Extended Onion}
\label{sec:onion_equivalence}

The second observation of Section~\ref{sec:cholesky_density} shows that each normalized-Bartlett row decomposes into a beta-distributed squared radius and an independent uniform direction on the sphere, precisely the law used by the extended-onion sampler. The relationship between the two samplers is sharper than distributional equivalence: the Gaussian vector $Z_i$ supplies two independent ingredients, its direction, which is uniform on the sphere, and its squared norm, which is chi-squared. The extended onion uses only the direction; the normalized-Bartlett sampler uses both.

\begin{theorem}[Exact row-wise coupling]
\label{thm:coupling}
For each row $i=2,\ldots,n$, let $Z_i\sim N_{i-1}(0,I_{i-1})$ and $G_i\sim\chi^2_{n+2\eta-i}$ be independent, and define
$$
Q_i=Z_i^\prime Z_i,\qquad
U_i=\frac{Z_i}{Q_i^{1/2}},\qquad
R_i^2=\frac{Q_i}{Q_i+G_i},\qquad
S_i=(Q_i+G_i)^{1/2}.
$$
Then:
\begin{itemize}
\item[(i)] $U_i$ is uniform on the unit sphere $\mathbb{S}^{i-2}\subset\mathbb{R}^{i-1}$, $Q_i\sim\chi^2_{i-1}$, and $U_i$ and $Q_i$ are independent;
\item[(ii)] $R_i^2\sim\operatorname{Beta}\bigl(\frac{i-1}{2},\eta+\frac{n-i}{2}\bigr)$, independent of $U_i$, so the pair $(R_i,U_i)$ has exactly the beta-radius/uniform-direction law used by the extended-onion sampler for row $i$;
\item[(iii)] the extended-onion row constructed from the radius $R_i$ and direction $U_i$ coincides \emph{almost surely} with the normalized-Bartlett row of Algorithm~\ref{alg:bartlett}:
$$
R_iU_i^\prime=\frac{Z_i^\prime}{S_i}=L_{i,1:i-1},
\qquad
(1-R_i^2)^{1/2}=\frac{G_i^{1/2}}{S_i}=L_{ii}.
$$
\end{itemize}
The couplings are independent across rows, and the coupled onion input $(X_i,Y_i)=(Q_i,G_i)$ has the law $\operatorname{Gamma}((i-1)/2,2)\otimes\operatorname{Gamma}(\eta+(n-i)/2,2)$ required by the gamma-ratio construction of the beta radius. Consequently, the gamma-ratio onion sampler driven by $(Q_i,G_i,U_i)$ produces the Algorithm~\ref{alg:bartlett} output pathwise, row by row.
\end{theorem}

\begin{proof}
Spherical symmetry of $Z_i$ implies that $U_i=Z_i/\|Z_i\|$ is uniform on $\mathbb{S}^{i-2}$ and independent of $Q_i=\|Z_i\|^2\sim\chi^2_{i-1}$, giving (i). In the scale-2 Gamma parametrization, $Q_i\sim\operatorname{Gamma}((i-1)/2,2)$ and $G_i\sim\operatorname{Gamma}(\eta+(n-i)/2,2)$ (the common scale cancels in the ratio), so the gamma-ratio identity gives $R_i^2=Q_i/(Q_i+G_i)\sim\operatorname{Beta}((i-1)/2,\eta+(n-i)/2)$; since $R_i^2$ is a function of $(Q_i,G_i)$, which is independent of $U_i$, (ii) follows. For (iii), $R_iU_i=(Q_i^{1/2}/S_i) (Z_i/Q_i^{1/2})=Z_i/S_i$ and $1-R_i^2=G_i/S_i^2$, which are the Algorithm~\ref{alg:bartlett} assignments. Independence across rows is as in Proposition~\ref{prop:cholesky_density}, so the joint law of the rows of $L$ is the product of the row-wise laws, consistent with the joint $\operatorname{LKJ}_n(\eta)$ law of Proposition~\ref{prop:lkj_wishart}.
\end{proof}

Theorem~\ref{thm:coupling} identifies the exact redundancy in the gamma-ratio implementation of the onion construction: the first Gamma component of the beta radius is already contained in the normal draws used for the direction. That implementation discards $\|Z_i\|$ in the normalization $U_i=Z_i/\|Z_i\|$ and then regenerates an independent copy of its law as a fresh Gamma draw; the normalized-Bartlett sampler reuses it, so only the second Gamma component $G_i$ must be drawn separately. Section~\ref{sec:ledger} makes the count precise, including the row-$2$ initialization used in practice, which differs from the row-wise description above.

\section{Coupling with the C-Vine}
\label{sec:vine_equivalence}

The C-vine sampler of \citet{LewandowskiKurowickaJoe:2009} parametrizes the correlation matrix by the partial correlations $\rho_{ji;1\cdots j-1}$, $1\leq j<i\leq n$, drawn independently with the symmetric law $\rho_{ji;1\cdots j-1}\overset{d}{=}2V-1$, $V\sim\operatorname{Beta}(\beta_j,\beta_j)$, where $\beta_j=\eta+(n-1-j)/2$ depends only on the tree index $j$. The Cholesky factor is an explicit product function of these parameters \citep[Appendix~A.4]{JoeKurowicka:2026}:
\begin{equation}
\label{eq:vine_cholesky}
L_{ij}=\rho_{ji;1\cdots j-1}\prod_{k=1}^{j-1}\bigl(1-\rho_{ki;1\cdots k-1}^2\bigr)^{1/2},
\qquad
L_{ii}=\prod_{k=1}^{i-1}\bigl(1-\rho_{ki;1\cdots k-1}^2\bigr)^{1/2},
\end{equation}
for $1\leq j<i\leq n$, and the parameters of row $i$ appear in no other row. The following theorem shows that the normalized-Bartlett row generates the entire vine row: each partial correlation is the corresponding normal deviate divided by the norm of the remaining tail of the row.

We first name the quantities involved. For each row $i=2,\ldots,n$, let $Z_i\sim N_{i-1}(0,I_{i-1})$ and $G_i\sim\chi^2_{n+2\eta-i}$ be independent, and define the \emph{tail sums} and \emph{normalized ratios}
\begin{equation}
\label{eq:tails}
T_{ij}=\sum_{k=j+1}^{i-1}Z_{ik}^2+G_i,
\quad j=0,1,\ldots,i-1,
\qquad
P_{ij}=\frac{Z_{ij}}{\bigl(Z_{ij}^2+T_{ij}\bigr)^{1/2}}=\frac{Z_{ij}}{T_{i,j-1}^{1/2}},
\quad j=1,\ldots,i-1.
\end{equation}
Thus $T_{ij}$ collects the squared entries of row $i$ lying beyond position $j$, together with the diagonal term $G_i$; the two extremes are $T_{i0}=S_i^2$, the whole row, and $T_{i,i-1}=G_i$. The tails are \emph{nested}, in the sense that $T_{i,j-1}=Z_{ij}^2+T_{ij}$: each is obtained from the next by restoring one entry, so consecutive tails share all but one term.

\begin{theorem}[Exact row-wise coupling with the C-vine]
\label{thm:vine_coupling}
Let $T_{ij}$ and $P_{ij}$ be as in \eqref{eq:tails}. Then, for each row $i=2,\ldots,n$:
\begin{itemize}
\item[(i)] the $P_{ij}$, $j=1,\ldots,i-1$, are mutually independent, and $P_{ij}\overset{d}{=}2V-1$ with $V\sim\operatorname{Beta}(\beta_j,\beta_j)$, $\beta_j=\eta+(n-1-j)/2$; this is exactly the tree-$j$ law of the C-vine sampler for $\operatorname{LKJ}_n(\eta)$;
\item[(ii)] the vine construction \eqref{eq:vine_cholesky} applied to $\rho_{ji;1\cdots j-1}=P_{ij}$ returns the normalized-Bartlett row \emph{almost surely}:
$$
P_{ij}\prod_{k=1}^{j-1}\bigl(1-P_{ik}^2\bigr)^{1/2}=\frac{Z_{ij}}{S_i}=L_{ij},
\qquad
\prod_{k=1}^{i-1}\bigl(1-P_{ik}^2\bigr)^{1/2}=\frac{G_i^{1/2}}{S_i}=L_{ii}.
$$
\end{itemize}
The couplings are independent across rows, so the array $\{P_{ij}\}$ has exactly the joint input law of the C-vine sampler, and the C-vine sampler driven by it produces the Algorithm~\ref{alg:bartlett} output pathwise.
\end{theorem}

For row $i=4$, for example, the theorem reads
$$
P_{43}=\frac{Z_{43}}{(Z_{43}^2+G_4)^{1/2}},
\qquad
P_{42}=\frac{Z_{42}}{(Z_{42}^2+Z_{43}^2+G_4)^{1/2}},
\qquad
P_{41}=\frac{Z_{41}}{(Z_{41}^2+Z_{42}^2+Z_{43}^2+G_4)^{1/2}},
$$
and the products in \eqref{eq:vine_cholesky} telescope because $1-P_{41}^2=(Z_{42}^2+Z_{43}^2+G_4)/S_4^2$ and so on down the row.

\begin{proof}
The squared entries $Z_{ij}^2\sim\operatorname{Gamma}(\tfrac12,2)$ and $G_i\sim\operatorname{Gamma}\bigl(\tfrac{n+2\eta-i}{2},2\bigr)$ are mutually independent, so each tail sum is Gamma with the summed shapes:
$$
T_{ij}\sim\operatorname{Gamma}\Bigl(\tfrac{i-1-j}{2}+\tfrac{n+2\eta-i}{2},2\Bigr)=\operatorname{Gamma}(\beta_j,2),
\qquad
\beta_j=\tfrac{n+2\eta-j-1}{2}=\eta+\tfrac{n-1-j}{2},
$$
which depends on $j$ only. Write $D_{ij}=P_{ij}^2=Z_{ij}^2/(Z_{ij}^2+T_{ij})$. Since $Z_{ij}^2$ and $T_{ij}$ are independent Gammas with common scale, the change of variables $(x_1,x_2)\mapsto(x_1/(x_1+x_2),x_1+x_2)$, whose Jacobian factorizes the joint density, gives $D_{ij}\sim\operatorname{Beta}(\tfrac12,\beta_j)$, independent of the sum $Z_{ij}^2+T_{ij}=T_{i,j-1}$. The mutual independence of $(D_{i1},\ldots,D_{i,i-1})$ now follows by backward induction on $j$, with hypothesis $H_j$: \emph{the variables $D_{i,j+1},\ldots,D_{i,i-1}$ are mutually independent and jointly independent of $T_{ij}$}. The base case $H_{i-1}$ is trivial. Suppose $H_j$ holds for some $j\in\{1,\ldots,i-1\}$. The fresh entry $Z_{ij}^2$ is independent of $(Z_{i,j+1}^2,\ldots,Z_{i,i-1}^2,G_i)$ and hence of $(D_{i,j+1},\ldots,D_{i,i-1},T_{ij})$, so the pair $(Z_{ij}^2,T_{ij})$ is jointly independent of $(D_{i,j+1},\ldots,D_{i,i-1})$. Both $D_{ij}$ and $T_{i,j-1}$ are functions of $(Z_{ij}^2,T_{ij})$, and the same change of variables gives $D_{ij}\perp T_{i,j-1}$; combining, $(D_{ij},T_{i,j-1})$ has independent components and is jointly independent of $(D_{i,j+1},\ldots,D_{i,i-1})$, which is $H_{j-1}$ with $D_{ij}$ adjoined. At $j=0$ all the $D_{ij}$ are mutually independent (and independent of $T_{i0}=S_i^2$).

For the signs, Gaussian symmetry makes $\operatorname{sign}(Z_{i1}),\ldots,\operatorname{sign}(Z_{i,i-1})$ i.i.d.\ uniform on $\{-1,1\}$ and independent of all squared magnitudes and of $G_i$; since $P_{ij}=\operatorname{sign}(Z_{ij})D_{ij}^{1/2}$, the $P_{ij}$ are mutually independent and symmetric with $P_{ij}^2\sim\operatorname{Beta}(\tfrac12,\beta_j)$. If $V\sim\operatorname{Beta}(\beta,\beta)$ then $2V-1$ is symmetric with $(2V-1)^2\sim\operatorname{Beta}(\tfrac12,\beta)$, and a distribution on $(-1,1)$ that is symmetric about zero with sign independent of magnitude is determined by the law of its square; hence $P_{ij}\overset{d}{=}2V-1$, proving (i).

For (ii), $1-P_{ij}^2=T_{ij}/T_{i,j-1}$, so the products telescope: $\prod_{k=1}^{j-1}(1-P_{ik}^2)=T_{i,j-1}/T_{i0}=T_{i,j-1}/S_i^2$. Therefore
$$
P_{ij}\prod_{k=1}^{j-1}\bigl(1-P_{ik}^2\bigr)^{1/2}
=\frac{Z_{ij}}{T_{i,j-1}^{1/2}}\cdot\frac{T_{i,j-1}^{1/2}}{S_i}
=\frac{Z_{ij}}{S_i},
\qquad
\prod_{k=1}^{i-1}\bigl(1-P_{ik}^2\bigr)^{1/2}=\frac{T_{i,i-1}^{1/2}}{S_i}=\frac{G_i^{1/2}}{S_i},
$$
matching Algorithm~\ref{alg:bartlett}. Independence across rows is again as in Proposition~\ref{prop:cholesky_density}.
\end{proof}

The coupling inverts explicitly. By the telescoping identity in the proof, $\prod_{k=1}^{j-1}(1-P_{ik}^2)=1-\sum_{k=1}^{j-1}L_{ik}^2$, so
$$
P_{ij}=\frac{L_{ij}}{\bigl(1-\sum_{k=1}^{j-1}L_{ik}^2\bigr)^{1/2}},
$$
and the vine parameters are recovered from the normalized factor alone, without access to the discarded scale $S_i$. The identity uses no property of the LKJ law, so it gives a direct algorithm for converting any positive-diagonal Cholesky factor of a positive definite correlation matrix into its C-vine partial correlations.

Theorem~\ref{thm:vine_coupling} has the following three implications. First, the coupling identifies the redundancy in the vine. Each partial correlation involves two independent Gamma-distributed variables: in the decomposition of the theorem, its $\chi^2_1$ numerator and its $\operatorname{Gamma}(\beta_j,2)$ tail; in the conventional symmetric-beta implementation, the two $\operatorname{Gamma}(\beta_j,2)$ components of $2V-1$. The two representations price differently. In the tail representation the numerator is supplied by the squared normal draw, so only the tail requires a Gamma call, and each parameter costs one normal plus one Gamma; in the conventional symmetric-Beta representation both components are drawn as Gammas, so each parameter costs two Gamma-equivalent calls. The coupling shows that the required Gamma components are the nested tail sums $T_{ij}=Z_{i,j+1}^2+T_{i,j+1}$, which overlap: consecutive parameters share all but one term. The fresh-tail vine implementation regenerates each tail independently from scratch; the Bartlett route draws only the increments, $i-1$ normals and a single chi-squared per row, and obtains every tail by subtraction-free accumulation.

Second, the Beta parameter $\beta_j$ depends only on the tree index $j$, not on the row $i$; this is the structural fact underlying the common tree-level distributions in \citet{LewandowskiKurowickaJoe:2009} and their moment recursions in \citet{JoeKurowicka:2026}. Under the coupling this is immediate: $T_{ij}$ has shape $\beta_j$ for every row $i$, since increasing $i$ by one adds a term of shape $1/2$ to the tail sum while reducing the shape of $G_i$ by $1/2$.

Third, the onion variables are functions of the vine variables: $R_i^2=1-\prod_{j<i}(1-P_{ij}^2)$ and $U_i^\prime=L_{i,1:i-1}/R_i$. Theorem~\ref{thm:coupling} is thus the aggregated form of Theorem~\ref{thm:vine_coupling}, in which the radius collects the whole row at once, while the vine coupling is the fully sequential form.

\section{Characterizations of the LKJ Family}
\label{sec:characterizations}

The results of Sections~\ref{sec:onion_equivalence} and~\ref{sec:vine_equivalence} lead to three complementary structural descriptions of the LKJ family. The first gives a variational interpretation of $\eta$, which follows from standard exponential-family duality, but we have not seen it stated in the existing literature. The second and third are known, and align directly with the vine and Wishart--Bartlett representations.

Expressed in its $d=n(n-1)/2$ free off-diagonal coordinates, $\mathcal{E}_n$ is a subset of the hypercube $[-1,1]^d$. All densities and differential entropies below are relative to Lebesgue measure $\lambda_d$ in those coordinates. Differential entropy is not invariant under reparametrization, so what follows is a statement about this flat off-diagonal reference measure, not a coordinate-free assertion about information on the elliptope.

\begin{proposition}[Maximum entropy]
\label{prop:maxent}
Let $\mathcal{P}$ be the set of probability distributions on $\mathcal{E}_n$ with a density with respect to $\lambda_d$, and for $P\in\mathcal{P}$ with density $f$ let $H(P)=-\int f\log fd\lambda_d=-\mathbb{E}_P[\log f]$ be the differential entropy. Define
$$
m(\eta)=\mathbb{E}_{\operatorname{LKJ}_n(\eta)}[\log\det C]
=\sum_{j=0}^{n-1}\psi\Bigl(\eta+\tfrac{j}{2}\Bigr)-n\psi\Bigl(\eta+\tfrac{n-1}{2}\Bigr),
$$
where $\psi$ is the digamma function. Then:
(i) for every $\eta>0$, $\operatorname{LKJ}_n(\eta)$ is the unique maximizer of $H(P)$ over $\{P\in\mathcal{P}:\mathbb{E}_P[\log\det C]=m(\eta)\}$;
(ii) $m$ is continuous and strictly increasing on $(0,\infty)$, with $m(\eta)\rightarrow-\infty$ as $\eta\rightarrow0$ and $m(\eta)\rightarrow0$ as $\eta\rightarrow\infty$, so every value in $(-\infty,0)$ is attained by exactly one $\eta$;
(iii) the uniform distribution, $\eta=1$, is the unique maximizer of $H(P)$ over all of $\mathcal{P}$, with no constraint imposed.
\end{proposition}

\begin{proof}
Write $g_\eta(C)=Z_n(\eta)^{-1}\det(C)^{\eta-1}$ for the $\operatorname{LKJ}_n(\eta)$ density, where $Z_n(\eta)=\int_{\mathcal{E}_n}\det(C)^{\eta-1}d\lambda_d$. For any $P\in\mathcal{P}$ with density $f$ and $\mathbb{E}_P[\log\det C]=m(\eta)$, the divergence inequality gives
$$
0\leq D(f\Vert g_\eta)=\int f\log\frac{f}{g_\eta}d\lambda_d
=-H(P)+\log Z_n(\eta)-(\eta-1)m(\eta),
$$
with equality if and only if $f=g_\eta$ almost everywhere. Since the same identity holds with equality at $f=g_\eta$, the right-hand side is $H(\operatorname{LKJ}_n(\eta))-H(P)$, proving (i). For (iii), $g_1\equiv Z_n(1)^{-1}$ is constant, so $0\leq D(f\Vert g_1)=-H(P)+\log Z_n(1)$ holds for every $P\in\mathcal{P}$, with no constraint imposed; hence $H(P)\leq H(\operatorname{LKJ}_n(1))$, uniquely attained at $f=g_1$.

For (ii), differentiate the closed form: $m^\prime(\eta)=\sum_{j=0}^{n-1}\psi_1(\eta+\frac{j}{2})-n\psi_1(\eta+\frac{n-1}{2})>0$, since the trigamma function $\psi_1$ is strictly decreasing and every argument in the sum is at most $\eta+\frac{n-1}{2}$, strictly so for $j<n-1$; the inequality is strict because $n\geq2$. Equivalently, exponential-family duality gives $m(\eta)=\frac{d}{d\eta}\log Z_n(\eta)$ and $m^\prime(\eta)=\operatorname{var}_{\operatorname{LKJ}_n(\eta)}(\log\det C)>0$, differentiation under the integral being justified for $\eta>0$ by local domination. The closed form itself follows from Proposition~\ref{prop:lkj_wishart}: with $\nu=n+2\eta-1$ and $W\sim W_n(\nu,I_n)$ we have $\log\det C=\log\det W-\sum_i\log W_{ii}$, where $\det W=\prod_i A_{ii}^2$ with $A_{ii}^2\sim\chi^2_{\nu-i+1}$ by Bartlett, and $W_{ii}\sim\chi^2_\nu$; the $\log2$ terms cancel, leaving $\sum_{i=1}^n\psi(\frac{\nu-i+1}{2})-n\psi(\frac{\nu}{2})$, which is the stated expression. The vine factorization $\det C=\prod_e(1-\rho_e^2)$ \citep{KurowickaCooke:2003} with Theorem~\ref{thm:vine_coupling}(i) gives the equivalent form $\sum_{\ell=1}^{n-1}(n-\ell)\{\log4+2\psi(\beta_\ell)-2\psi(2\beta_\ell)\}$, $\beta_\ell=\eta+(n-1-\ell)/2$, and the agreement of the two checks the representations against each other. The limits follow from $\psi(\eta)\rightarrow-\infty$ as $\eta\rightarrow0$ and $\psi(\eta+a)-\psi(\eta+b)\rightarrow0$ as $\eta\rightarrow\infty$.
\end{proof}

The natural parameter $\eta-1$ is therefore the Lagrange multiplier dual to the expected log-determinant, and $\eta$ indexes that constraint one-to-one through $m$. Relative to the flat reference measure $\lambda_d$, $\operatorname{LKJ}_n(\eta)$ is the maximum-entropy distribution at fixed $\mathbb{E}[\log\det C]$; beyond the support and this single moment constraint, no further feature of the family is imposed.

The interest of this is that it supplies an interpretation of $\eta$ that the defining density does not. Written as $\det(C)^{\eta-1}$, the parameter is a shape index whose effect is normally learned by simulation: one draws matrices at several values and inspects the resulting marginals. Part~(ii) replaces that with an exact correspondence. Since $m$ is a continuous, strictly increasing bijection onto $(-\infty,0)$, every attainable value of $\mathbb{E}[\log\det C]$ is matched by exactly one $\eta$, so choosing $\eta$ and choosing an expected log-determinant are the same act, and the latter is directly interpretable: $\log\det C$ measures how far the matrix is from singular, and equals $2\sum_{i\geq2}\log L_{ii}$ in the Cholesky coordinates the samplers of this paper produce. Part~(iii) places the uniform law in the same picture as the endpoint at which no constraint is imposed at all. Together these give the LKJ family the usual justification of a default prior: among all distributions on $\mathcal{E}_n$ with a given expected log-determinant, it is the one that assumes least. At $\eta=1$ the formula for $m$ agrees with the digamma expression for the uniform law in \citet{Hansen:CorrelationMatrix}.

The second description is the vine characterization. Fix a regular vine on $n$ variables, let $\rho_e$ denote its partial correlations and $\ell(e)$ the tree index of edge $e$. Under $\operatorname{LKJ}_n(\eta)$ the $\rho_e$ are mutually independent with $\rho_e\overset{d}{=}2V-1$, $V\sim\operatorname{Beta}(\beta_{\ell(e)},\beta_{\ell(e)})$, and the statement holds for every choice of regular vine \citep[Theorem~1]{LewandowskiKurowickaJoe:2009}. It also characterizes the family: the map from vine partial correlations to $C$ is a bijection onto the open elliptope $\mathcal{E}_n^\circ$ \citep{BedfordCooke:2002,KurowickaCooke:2003}, so a distribution on $\mathcal{E}_n^\circ$ whose partial correlations along a single regular vine are mutually independent with these laws must be $\operatorname{LKJ}_n(\eta)$. Theorem~\ref{thm:vine_coupling} is the algorithmic form: for each Cholesky row, it realizes that row's independent symmetric-Beta C-vine inputs across all tree levels at once, from the nested tails of a single Gaussian row.

The third is the spherical representation. By Proposition~\ref{prop:lkj_wishart}, $C\sim\operatorname{LKJ}_n(\eta)$ exactly when $C$ is the correlation matrix of $W\sim W_n(\nu,I_n)$ with $\nu=n+2\eta-1$. For integer $\nu\geq n$ this takes a Gram form, $C\overset{d}{=}V^\prime V$ with independent columns $v_i=z_i/\|z_i\|$, where the $z_i$ are independent and spherically distributed on $\mathbb{R}^\nu$, possibly with different radial laws and with no atom at the origin; sphericity makes each $v_i$ uniform on the sphere whatever its radius, so the law of $C$ is free of the radial distributions. Through Bartlett factorization, Algorithm~\ref{alg:bartlett} provides the real-$\nu$ continuation of this Gaussian Gram construction. The invariance is exact rather than distributional, since $(Qv_i)^\prime(Qv_j)=v_i^\prime v_j$: a simultaneous rotation of the configuration leaves $C$ unchanged pointwise. It is therefore attached to the ambient sample space of the representation, not to the elliptope, where generic orthogonal conjugations fail to preserve the unit diagonal and the signed permutations, which do preserve it, are far too weak to single out the family.

Analytically, then, the family is the log-determinant exponential family on $\mathcal{E}_n$, equivalently the maximum-entropy family at fixed $\mathbb{E}[\log\det C]$ relative to $\lambda_d$; structurally it is the family with independent symmetric-Beta vine partial correlations; and, when $\nu=n+2\eta-1$ is an integer, it is the Gram matrix of independent spherical directions, with the Bartlett construction providing the continuation to all real $\nu>n-1$. The latter two align directly with the C-vine and normalized-Bartlett samplers.

\section{Sampling Cost and Numerical Stability}
\label{sec:ledger}

This section compares the scalar distribution calls made by the four constructions of Table~\ref{tab:rng_ledger}.

\begin{table}[H]
\centering
\caption{Representation-level distribution-call ledger for $\operatorname{LKJ}_n(\eta)$ Cholesky sampling. The extended-onion row describes the algorithm of \citet{LewandowskiKurowickaJoe:2009} as implemented in \texttt{Distributions.jl}, including the beta-only initialization of row 2. The Gamma-equivalent column applies the gamma-ratio accounting (beta $=2$, chi-squared $=1$) and is representation-specific.}
\label{tab:rng_ledger}
\begin{tabularx}{\textwidth}{lYYYY}
\toprule
Sampler & Normal & Beta & Chi-squared & Gamma-equivalent \\
\midrule
Extended onion & $\dfrac{n(n-1)}{2}-1$ & $n-1$ & $0$ & $2(n-1)$ \\[6pt]
C-vine, conventional & $0$ & $\dfrac{n(n-1)}{2}$ & $0$ & $n(n-1)$ \\[6pt]
C-vine, tail representation & $\dfrac{n(n-1)}{2}$ & $0$ & $\dfrac{n(n-1)}{2}$ & $\dfrac{n(n-1)}{2}$ \\[6pt]
Normalized Bartlett & $\dfrac{n(n-1)}{2}$ & $0$ & $n-1$ & $n-1$ \\
\bottomrule
\end{tabularx}
\end{table}

We use the following two accounting conventions: First, the ledger describes the extended-onion algorithm as specified by \citet[Section~3.2]{LewandowskiKurowickaJoe:2009} and implemented in \texttt{Distributions.jl},\footnote{File \texttt{src/cholesky/lkjcholesky.jl} in \texttt{Distributions.jl} v0.25.129, the version benchmarked in Section~\ref{sec:numerical}.} including its special initialization of row~$2$: that row is generated as $L_{21}=2V-1$ with $V\sim\operatorname{Beta}(\beta_0,\beta_0)$, $\beta_0=\eta+(n-2)/2$, which uses one beta draw and no normal draw. (This is distributionally consistent with the row-wise description of Section~\ref{sec:onion_equivalence}: if $V\sim\operatorname{Beta}(\beta_0,\beta_0)$ then $2V-1$ is symmetric with $(2V-1)^2\sim\operatorname{Beta}(1/2,\beta_0)$, matching a random sign times the beta radius.) Second, a chi-squared draw is counted as one Gamma-equivalent draw, via $\chi^2_a\equiv\operatorname{Gamma}(a/2,2)$, and a beta draw as two, via the gamma-ratio representation
$$
\operatorname{Beta}(\alpha,\beta)
\equiv
\frac{X}{X+Y},
\qquad
X\sim\operatorname{Gamma}(\alpha,1),
\qquad
Y\sim\operatorname{Gamma}(\beta,1),
$$
with $X$ and $Y$ independent.

A simple cost model prices these calls. Let $c_N$ and $c_\Gamma$ denote the average costs of a standard normal and a Gamma-equivalent draw under the gamma-ratio representation, and write $r=c_\Gamma/c_N$. The RNG-only model costs, retaining only scalar distribution calls, are
$$
\widetilde{T}_{\operatorname{onion}}(n):=\Bigl(\tfrac{n(n-1)}{2}-1\Bigr)c_N+2(n-1)c_\Gamma,
\qquad
\widetilde{T}_{\operatorname{vine}}(n):=n(n-1)c_\Gamma,
$$
$$
\widetilde{T}_{\operatorname{vine}}^{\operatorname{tail}}(n):=\tfrac{n(n-1)}{2}(c_N+c_\Gamma),
\qquad
\widetilde{T}_{\operatorname{Bartlett}}(n):=\tfrac{n(n-1)}{2}c_N+(n-1)c_\Gamma .
$$

\begin{proposition}[Distribution-call counts and cost ratios]
\label{prop:rng_ledger}
Under the accounting conventions above:
\begin{itemize}
\item[(i)] the distribution-call counts of the four constructions are as in Table~\ref{tab:rng_ledger};
\item[(ii)] under the cost model, the speed ratios relative to normalized Bartlett, $S=\widetilde{T}/\widetilde{T}_{\operatorname{Bartlett}}$, are
$$
S_{\operatorname{onion}}(n;r)=\frac{n+4r-\frac{2}{n-1}}{n+2r},
\qquad
S_{\operatorname{vine}}(n;r)=\frac{2rn}{n+2r},
\qquad
S_{\operatorname{vine}}^{\operatorname{tail}}(n;r)=\frac{(1+r)n}{n+2r},
$$
so that $S_{\operatorname{onion}}\to1$, $S_{\operatorname{vine}}\to2r$, and $S_{\operatorname{vine}}^{\operatorname{tail}}\to1+r$ as $n\to\infty$.
\end{itemize}
\end{proposition}

\begin{proof}
For (i): extended onion: row $2$ uses one beta draw and no normals; each row $i=3,\ldots,n$ uses one beta draw and $i-1$ normals for the direction, giving $\sum_{i=3}^n(i-1)=n(n-1)/2-1$ normals and $n-1$ beta draws. C-vine, conventional: one shifted symmetric beta draw per partial correlation, $n(n-1)/2$ in total, and every other operation is arithmetic. C-vine, tail representation: by Theorem~\ref{thm:vine_coupling}, each partial correlation is $Z/(Z^2+Y)^{1/2}$ with a fresh standard normal $Z$ (supplying both sign and $\chi^2_1$ numerator) and a fresh tail $Y\sim\operatorname{Gamma}(\beta_j,2)$, giving $n(n-1)/2$ normals and $n(n-1)/2$ Gamma draws. Algorithm~\ref{alg:bartlett}: each row $i=2,\ldots,n$ uses $i-1$ normals and one chi-squared draw, giving $n(n-1)/2$ normals and $n-1$ chi-squared draws. Row $1$ is deterministic in every case. Applying the accounting convention (beta $=2$, chi-squared $=1$ Gamma-equivalent) gives the counts in Table~\ref{tab:rng_ledger}. For (ii), divide each $\widetilde{T}$ by $\widetilde{T}_{\operatorname{Bartlett}}(n)$ and simplify; the limits follow.
\end{proof}

In words, relative to the onion the Bartlett route trades one extra normal for $n-1$ saved Gamma-equivalent draws; relative to the conventional vine implementation it saves $(n-1)^2$ Gamma-equivalent draws in aggregate, approximately two per partial correlation; and relative to the tail representation, whose normal count matches Algorithm~\ref{alg:bartlett} exactly, it reduces the Gamma draws from $n(n-1)/2$ to $n-1$. The two comparisons behave differently in dimension. Against the onion the saving is $O(n)$ against $O(n^2)$ shared work, so the advantage is most pronounced at small $n$. Against the vine the saving is proportional to the total work, so the advantage persists with dimension; in the cost model, the limiting ratios are approximately $6$ for the conventional implementation and $4$ for the tail representation when $r\approx3$ (Section~\ref{sec:numerical}). The constant is representation- and implementation-dependent; on the benchmarked platform, the replacement of normal draws by beta or Gamma draws produces the persistent advantage seen in Tables~\ref{tab:controlled_vine_thm} and~\ref{tab:controlled_vine}.

\begin{remark}
\label{rem:scope}
Note that ``Gamma-equivalent'' is an accounting device tied to the gamma-ratio representation: a beta sampler need not generate two Gamma variates, since rejection algorithms, transformations, and parameter-dependent branches are common, and a rejection method consumes a random number of deviates per draw. Proposition~\ref{prop:rng_ledger}(i) therefore compares distribution calls under a fixed representation, not PRNG invocations, random bits, or execution cost. The limits in Proposition~\ref{prop:rng_ledger}(ii) are likewise properties of the cost model, not of the implementations: the operations the model omits, product, transformation, memory, and tail accumulation, are themselves $O(n^2)$. The model captures the qualitative distinction between a vanishing onion advantage and a persistent vine advantage; $r$ is calibrated by the microbenchmarks of Section~\ref{sec:numerical}, which benchmark the representation directly and report the package-level comparison separately.
\end{remark}

\paragraph{In-place implementation.}
Algorithm~\ref{alg:bartlett} never forms the dense Wishart matrix $W$, and Proposition~\ref{prop:cholesky_density} shows that the correlation matrix $C=LL^\prime$ need not be formed either: the row-normalized Bartlett factor is the LKJ Cholesky factor. The construction therefore admits a fully in-place implementation, stated as Algorithm~\ref{alg:bartlett_inplace}. For each row $i$, the $i-1$ standard normals are written directly into $L_{i,1:i-1}$, one chi-squared variate is drawn, the squared row norm is accumulated, and the row is overwritten by its normalized values. The method requires $O(n^2)$ output storage and $O(1)$ additional scalar workspace, with no per-row heap allocation.

Because only the lower triangle is written, the storage contract must be explicit: the output must be either (a) a triangular or packed-triangular type whose upper part is ignored by construction (e.g., a \texttt{LowerTriangular} view or a \texttt{Cholesky} factorization object, the convention used by \texttt{Distributions.jl}); or (b) a dense array whose strictly upper triangle is zeroed, once, before or after the row loop. Algorithm~\ref{alg:bartlett_inplace} adopts convention (b) for concreteness; the zeroing is a one-time $O(n^2)$ pass that can be skipped when the caller reuses a buffer whose upper triangle is already zero, or when the result is wrapped in a triangular type.

\begin{algorithm}[H]
\caption{In-place normalized-Bartlett $\operatorname{LKJ}_n(\eta)$ Cholesky sampler}
\label{alg:bartlett_inplace}
\begin{algorithmic}[1]
\Require preallocated $n\times n$ array $L$; parameter $\eta>0$
\State $L_{ij}\gets0$ for all $i<j$ \Comment{skip if buffer's upper triangle is known to be zero, or if $L$ is a triangular type}
\State $L_{11}\gets1$
\For{$i=2,\ldots,n$}
  \State Fill $L_{i,1:i-1}$ with independent $N(0,1)$ draws \Comment{$i-1$ normals, written in place}
  \State $q\gets\sum_{j=1}^{i-1}L_{ij}^2$ \Comment{squared row norm}
  \State Draw $g\sim\chi^2_{n+2\eta-i}$
  \State $s\gets(q+g)^{1/2}$
  \State $L_{i,1:i-1}\gets L_{i,1:i-1}/s$ \Comment{overwrite row with normalized values}
  \State $L_{ii}\gets g^{1/2}/s$
\EndFor
\State \Return $L$ \Comment{$LL^\prime\sim\operatorname{LKJ}_n(\eta)$; neither $W$ nor $C$ is formed}
\end{algorithmic}
\end{algorithm}

\paragraph{Numerical behavior.}
The normalized-Bartlett form contains no subtraction. The diagonal entry is computed as $L_{ii}=g^{1/2}/s$ from nonnegative quantities, so it is computed to small relative error whenever the drawn variates are strictly positive and representable in the floating-point format. By contrast, the textbook onion step computes $L_{ii}=(1-R_i^2)^{1/2}$, which suffers cancellation when $R_i^2$ is close to one: an absolute error of order machine epsilon in $R_i^2$ becomes a relative error of order $\varepsilon/(1-R_i^2)$ in $1-R_i^2$. The regime $R_i^2\approx1$ arises precisely when $G_i$ is small relative to $Q_i$, which is most likely in the late rows when $\eta$ is small: for the last row, $G_n\sim\chi^2_{2\eta}$ concentrates near zero as $\eta\downarrow0$. The numerical advantage is not intrinsic to the onion or C-vine distributions. Under the gamma-ratio couplings, however, the subtraction-free implementations are exactly the normalized-Bartlett assignments: $L_{ii}=\{G_i/(Q_i+G_i)\}^{1/2}$ for the onion (Theorem~\ref{thm:coupling}), and the corresponding nested tail ratios for the vine (Theorem~\ref{thm:vine_coupling}). Within these representations the Bartlett formulation makes the stable computation immediate, whereas the textbook formulas, the current onion package implementation, and the conventional C-vine implementation examined here recover complements through $1-R_i^2$ or $1-\rho^2$, where cancellation can occur.

Two qualifications apply. First, for very small shape parameters the chi-squared draw itself can underflow to zero in finite precision, yielding $L_{nn}=0$ and a singular factor. The textbook computations face the mirror-image problem, $1-R_n^2$ rounding to zero, at a much larger threshold: under ideal, correctly rounded generation the Bartlett diagonal vanishes only when $G_n$ falls below roughly half the smallest positive subnormal value ($\approx10^{-324}$), whereas the subtractive forms lose $1-R_n^2$ once it falls below the rounding threshold at one, $2^{-54}\approx5.6\times10^{-17}$, since an exact $1-\delta$ rounds to unity below that point; realized thresholds can additionally depend on the internals of the Gamma and beta generators. The leading-order failure probabilities are $\Pr(G_n<t)\sim(t/2)^{\eta}/\Gamma(\eta+1)$ and $\Pr(1-R_n^2<\delta)\sim\delta^{\eta}/\{\eta B(\eta,(n-1)/2)\}$, the latter with an $n$-dependent constant; \xtab{tab:stress}{S1} compares these predictions with observed frequencies, which agree to about two digits. Second, the computed row satisfies $\sum_{j\leq i}L_{ij}^2=1$ up to a few units in the last place; the normalization is a single division by a computed norm, not an exact projection. In double precision both methods are adequate for routine use.

\section{Benchmarks}
\label{sec:numerical}

The benchmarks in this section are of two kinds, and they answer different questions. The \emph{controlled comparisons} isolate the mechanisms of the coupling theorems: for the onion (Theorem~\ref{thm:coupling}), two samplers written in the same style differing only in whether the first Gamma component of the beta radius is drawn or reused; for the vine (Theorem~\ref{thm:vine_coupling}), samplers identical except for whether each partial correlation's Gamma tail is drawn fresh or computed from the shared row sums, with the conventional two-Gamma symmetric-beta implementation benchmarked separately at the representation level. The \emph{package-level comparison} measures complete implementations, our samplers against \texttt{Distributions.jl}, and is presented as such: its speedups reflect loop structure, beta-sampler internals, allocation, and interface overhead in addition to the saved draws. The package-level comparison is confined to Julia, where \texttt{Distributions.jl} provides an onion implementation. Stan implements the partial-correlation route, and its Beta generator uses two Gamma draws per partial correlation, matching the Gamma-call count of Table~\ref{tab:rng_ledger}, although its small-shape branch uses an augmented log-space construction rather than the direct Gamma ratio. A direct timing comparison across Stan Math and Julia would confound the algorithmic difference with language, RNG-library, and interface effects; the vine comparisons are therefore kept controlled and representation-level.

\paragraph{Protocol.}
All benchmarks use \texttt{BenchmarkTools.jl} \citep{ChenRevels:2016} with fixed seeds. We report median times; interquartile ranges, minima, and allocation counts are recorded alongside every median in the repository output for comparability with common practice. Medians are used because part of the run-to-run variability is intrinsic (rejection-based Gamma and beta samplers consume a random number of deviates), so the minimum is not an unbiased proxy for typical cost. Interface parity is enforced in both directions: allocating calls are compared with allocating calls (both returning a Cholesky-factor object), and preallocated calls with preallocated calls (\texttt{rand!} for \texttt{Distributions.jl} against Algorithm~\ref{alg:bartlett_inplace} with the zeroing pass skipped, since \texttt{Distributions.jl} likewise writes only one triangle of its \texttt{Cholesky} output). Distribution objects are constructed outside the timed loops.

\begin{table}[t]
\centering
\caption{Controlled comparison, $\eta=1$: median time per draw (ns) for structurally identical samplers differing only in whether the first Gamma component of the beta radius is drawn independently (gamma-ratio onion) or reused as $Q_i=Z_i^\prime Z_i$ under the coupling of Theorem~\ref{thm:coupling}. Both samplers use the Gaussian-direction representation in every row, including row~2, and both deliberately retain the same onion arithmetic, including the $(1-R_i^2)^{1/2}$ diagonal assignment, so that the timing difference isolates the extra Gamma draw. Neither column is the direct normalized-Bartlett implementation of Algorithm~\ref{alg:bartlett}, which is benchmarked in Table~\ref{tab:benchmarks}. The RNG-only model column is $(n+4r)/(n+2r)$ with $r=3.0$ calibrated from the microbenchmarks. Full grid over $\eta$ and dispersion measures in the code repository.}
\label{tab:controlled}
\begin{tabular}{rrrrr}
\toprule
$n$ & \makecell{Gamma-ratio\\onion (ns)} & \makecell{Bartlett-coupled\\reuse (ns)} & Ratio & \makecell{RNG-only\\model} \\
\midrule
  3 &     64.9 &     37.8 & 1.72 & 1.67 \\
  5 &    157.8 &    111.3 & 1.42 & 1.55 \\
 10 &    423.5 &    329.1 & 1.29 & 1.38 \\
 20 &  1{,}233 &  1{,}028 & 1.20 & 1.23 \\
 50 &  6{,}217 &  5{,}667 & 1.10 & 1.11 \\
100 & 23{,}167 & 22{,}041 & 1.05 & 1.06 \\
200 & 88{,}667 & 86{,}417 & 1.03 & 1.03 \\
\bottomrule
\end{tabular}
\end{table}

\paragraph{Controlled comparison: onion.}
Table~\ref{tab:controlled} reports the experiment that isolates Theorem~\ref{thm:coupling}: both samplers draw the same normals and the same second Gamma component; the onion variant additionally draws the first Gamma component that the theorem shows to be redundant. Both use the Gaussian-direction representation in every nontrivial row, including row~2, so that the two code paths differ only by the extra Gamma component; this departs from the Beta-only row-2 initialization counted in Proposition~\ref{prop:rng_ledger}. Any timing difference is attributable to the $n-1$ saved Gamma draws, because everything else in the two code paths is identical. The RNG-only model predicts the ratio $(n+4r)/(n+2r)$, which is exact within the cost model for this Gaussian-direction variant, the row-$2$ correction in Proposition~\ref{prop:rng_ledger} not applying; with the microbenchmarked value $r\approx3.0$, the model tracks the observed ratios across the entire range of $n$, with maximum deviation about $0.13$ (at $n=5$) and agreement within $0.04$ for $n\geq20$. The mechanism of Theorem~\ref{thm:coupling}, priced by the ledger of Proposition~\ref{prop:rng_ledger}, thus accounts for the controlled timing differences essentially in full.

\paragraph{Controlled comparison: vine.}
Table~\ref{tab:controlled_vine_thm} reports the experiment that isolates Theorem~\ref{thm:vine_coupling}. Both samplers draw one standard normal per partial correlation and build the factor with the same product-accumulation code; the only difference is the source of each Gamma tail. The fresh variant draws a new $\operatorname{Gamma}(\beta_j,2)$ tail for every parameter (the tail representation of the theorem), while the shared-tail variant computes the nested tails from the shared row sums, consuming one chi-squared draw per row. For control, both columns form the Cholesky factor through the same C-vine product map; the shared-tail column is distributionally and pathwise equivalent to Algorithm~\ref{alg:bartlett} but is not the direct implementation of Algorithm~\ref{alg:bartlett}, and it is correspondingly slower than the direct Algorithm~\ref{alg:bartlett_inplace} implementation of Table~\ref{tab:benchmarks} ($123{,}083$~ns against $85{,}541$~ns at $n=200$). The timing difference between the two columns isolates the replacement of fresh per-parameter Gamma tails by row-level chi-squared draws and shared-tail accumulation, and the RNG-only model predicts a ratio approaching $1+r\approx4$. That model is a rough benchmark rather than a precise prediction, for the reason given in Remark~\ref{rem:scope}.
The observed ratios rise from $1.3\times$ at $n=3$ to a plateau of about $2.5$--$2.7\times$ for $n\geq50$: the persistence the ledger predicts is clearly visible, while the plateau sits well below the $1+r\approx4$ benchmark and does not approach it as $n$ grows, since the omitted per-entry arithmetic scales with the priced work rather than vanishing relative to it.

\begin{table}[t]
\centering
\caption{Controlled comparison for the vine coupling, $\eta=1$: median time per draw (ns) for samplers identical except for the source of each partial correlation's Gamma tail, drawn fresh per parameter (tail representation of Theorem~\ref{thm:vine_coupling}) or computed from the shared row sums (Bartlett-coupled shared tails). Both columns form the factor through the same C-vine product map. The RNG-only model column is $(1+r)n/(n+2r)$ with $r=3.0$ from the microbenchmarks.}
\label{tab:controlled_vine_thm}
\begin{tabular}{rrrrr}
\toprule
$n$ & \makecell{Fresh tails\\(ns)} & \makecell{Bartlett-coupled\\shared tails (ns)} & Ratio & \makecell{RNG-only\\model} \\
\midrule
  3 &     60.2 &     47.0 & 1.28 & 1.33 \\
  5 &    206.9 &    115.4 & 1.79 & 1.82 \\
 10 &    829.7 &    394.1 & 2.11 & 2.50 \\
 20 &  3{,}161 &  1{,}329 & 2.38 & 3.08 \\
 50 & 19{,}792 &  7{,}927 & 2.50 & 3.57 \\
100 & 80{,}125 & 30{,}292 & 2.65 & 3.77 \\
200 & 320{,}292 & 123{,}083 & 2.60 & 3.88 \\
\bottomrule
\end{tabular}
\end{table}

\paragraph{Representation-level comparison: conventional vine.}
Table~\ref{tab:controlled_vine} compares the conventional symmetric-beta implementation, which draws each partial correlation as $2X/(X+Y)-1$ with two fresh $\operatorname{Gamma}(\beta_j,2)$ variates, against the Bartlett-coupled shared-tail sampler with the same downstream code. This is a comparison of distribution-call representations rather than a direct implementation of the pathwise coupling: the two-Gamma decomposition here differs from the $\chi^2_1$-plus-tail decomposition of Theorem~\ref{thm:vine_coupling}, although the output laws coincide. The observed ratios rise from $1.8\times$ at $n=3$ and remain between roughly $4.8\times$ and $6.3\times$ for $n\geq20$; they support a persistent constant-factor advantage of roughly five to six for this implementation and are broadly consistent with the $2r\approx6$ limit of the RNG-only model, without confirming a precise limit. Deviations from that benchmark reach about $41\%$ at $n=5$ and $35\%$ at $n=10$, and lie within about $21\%$ for $n\geq20$, in either direction; they reflect accumulation, transformation, and memory costs outside the distribution-call ledger. The shared-tail columns of Tables~\ref{tab:controlled_vine_thm} and~\ref{tab:controlled_vine} time separately compiled instances of the same algorithm, so that each experiment remains internally paired; they agree within about $3\%$ for $n\geq20$, while at $n=5$ and $n=10$ a difference between the two compiled instances is visible at this timescale ($115$ against $102$~ns, and $394$ against $359$~ns), reproduces across seeds, and appears to reflect code-generation differences rather than sampling noise.

\begin{table}[t]
\centering
\caption{Representation-level comparison for the conventional C-vine implementation, $\eta=1$: median time per draw (ns) for the symmetric-beta gamma-ratio sampler ($2X/(X+Y)-1$, two fresh Gamma draws per partial correlation) against the Bartlett-coupled shared-tail sampler, identical downstream code. Both columns form the factor through the same C-vine product map. The RNG-only model column is $2rn/(n+2r)$ with $r=3.0$ from the microbenchmarks.}
\label{tab:controlled_vine}
\begin{tabular}{rrrrr}
\toprule
$n$ & \makecell{Gamma-ratio\\vine (ns)} & \makecell{Bartlett-coupled\\shared tails (ns)} & Ratio & \makecell{RNG-only\\model} \\
\midrule
  3 &     90.8 &     50.3 & 1.81 & 2.00 \\
  5 &    392.7 &    102.1 & 3.85 & 2.73 \\
 10 &  1{,}817 &    359.2 & 5.06 & 3.75 \\
 20 &  7{,}451 &  1{,}333 & 5.59 & 4.62 \\
 50 & 48{,}833 &  7{,}719 & 6.33 & 5.36 \\
100 & 152{,}042 & 30{,}208 & 5.03 & 5.66 \\
200 & 597{,}875 & 124{,}209 & 4.81 & 5.83 \\
\bottomrule
\end{tabular}
\end{table}

\paragraph{Microbenchmarks.}
To calibrate $r$ directly, we microbenchmark single scalar draws over the shape ranges that occur in the samplers. On the benchmark machine, a standard normal costs $c_N\approx3.8$~ns. A $\operatorname{Gamma}(\alpha,2)$ draw costs $\approx11$--$12$~ns for $\alpha>1$, $\approx6$~ns at $\alpha=1$ (the exponential fast path), and $\approx19$~ns for $\alpha<1$ (the small-shape branch); chi-squared draws behave identically at matching shapes. A $\operatorname{Beta}(\alpha,\beta)$ draw costs $\approx22$--$37$~ns depending on the shape pair. Taking $r=c_\Gamma/c_N\approx3.0$ for the moderate shapes that dominate the row loop gives the model column of Table~\ref{tab:controlled}; the small-$n$, small-$\eta$ corner is slightly more expensive because shapes below one trigger the slow branch. We deliberately do not back out cost ratios from the total execution times of the full samplers, since those totals include norm computations, scaling loops, dispatch, and allocation that the ledger does not model.

\paragraph{Package-level comparison.}
Table~\ref{tab:benchmarks} compares complete implementations: the \texttt{LKJCholesky} onion sampler in \texttt{Distributions.jl} against Algorithms~\ref{alg:bartlett} and~\ref{alg:bartlett_inplace}, across $n\in\{3,5,10,20,50,100,200\}$, with matched interfaces in both directions; these corroborate the low-dimensional advantage first reported by \citet{WangWuChu:2018}. All quantities in the table refer to $\eta=1$; the full grid over $\eta\in\{0.5,1,2,10\}$ is in the code repository. For $n\geq10$, timings are nearly identical across $\eta$ at fixed $n$, since $\eta$ does not alter the number of variates drawn; at the smallest dimensions the shape parameters $n+2\eta-i$ determine which branches of the underlying Gamma and beta samplers are triggered (see the microbenchmarks), so timings vary with $\eta$. The package-level speedups exceed those of the controlled comparison ($2.8\times$ against $1.3\times$ at $n=10$) because they also reflect implementation differences: beta-sampler internals, method dispatch, and per-call overhead in \texttt{Distributions.jl}. They should be read as a comparison of software, with Theorem~\ref{thm:coupling} explaining only the algorithmic component isolated in Table~\ref{tab:controlled}.

\begin{table}[t]
\centering
\caption{Package-level comparison at $\eta=1$: median time per draw (ns) and speedup relative to \texttt{Distributions.jl}, for matched interfaces: allocating calls returning a \texttt{Cholesky} object (left), and preallocated calls (\texttt{rand!} versus Algorithm~\ref{alg:bartlett_inplace}, right).
Benchmarks run in Julia~1.12.7 \citep{Bezanson:2017} (\texttt{Distributions.jl}~v0.25.129, \texttt{BenchmarkTools.jl}~v1.8.0 \citep{ChenRevels:2016})
on a 10-core Apple M1~Max (macOS, arm64), single thread, seed 19100618.}
\label{tab:benchmarks}
\begin{tabular}{rrrrrrr}
\toprule
& \multicolumn{3}{c}{Allocating} & \multicolumn{3}{c}{Preallocated} \\
\cmidrule(lr){2-4}\cmidrule(lr){5-7}
$n$ & Onion (ns) & Bartlett (ns) & Speedup & Onion (ns) & Bartlett (ns) & Speedup \\
\midrule
  3 &    178.9 &     54.0 & 3.31$\times$ &    166.1 &     37.3 & 4.45$\times$ \\
  5 &    395.8 &    123.6 & 3.20$\times$ &    386.5 &    106.3 & 3.64$\times$ \\
 10 &    996.4 &    355.0 & 2.81$\times$ &    944.0 &    318.2 & 2.97$\times$ \\
 20 &  2{,}389 &  1{,}096 & 2.18$\times$ &  2{,}375 &  1{,}030 & 2.31$\times$ \\
 50 & 10{,}250 &  5{,}722 & 1.79$\times$ & 10{,}250 &  5{,}611 & 1.83$\times$ \\
100 & 34{,}458 & 22{,}167 & 1.55$\times$ & 34{,}000 & 21{,}709 & 1.57$\times$ \\
200 & 124{,}292 & 86{,}000 & 1.45$\times$ & 123{,}250 & 85{,}541 & 1.44$\times$ \\
\bottomrule
\end{tabular}
\end{table}

\paragraph{Code availability.}
Julia code reproducing all benchmarks and tests, including \texttt{Project.toml} and \texttt{Manifest.toml}, the exact scripts generating each table, RNG types and seeds, and machine and thread settings, is available at \repourl. The numbers reported here correspond to commit \repocommit, released as tag \repotag{} and archived at \repodoi.

\paragraph{Implementation tests.}
All implementations are subjected to the regression and stress tests reported in \testslocation: pathwise coupling tests for both theorems, exact-law Kolmogorov--Smirnov tests of the row radius law, moment checks with Monte Carlo standard errors, separate regression tests for each of the three vine code paths, and stress tests at extreme parameters that confirm the underflow-threshold asymmetry of Section~\ref{sec:ledger} at the predicted rates.

\section{Conclusion}\label{sec:conclusion}

The two classical samplers for the LKJ distribution, the extended onion and the partial correlation C-vine of \citet{LewandowskiKurowickaJoe:2009}, admit exact pathwise realizations from a third, simpler object: the row-normalized Bartlett factor of a restricted Wishart matrix \citep{WangWuChu:2018}. Theorem~\ref{thm:coupling} couples the Bartlett row to the onion's radius--direction pair; Theorem~\ref{thm:vine_coupling} couples the same row to the full set of vine partial correlations, each recovered as a normal deviate divided by the norm of the remaining tail, with exactly the independent symmetric-Beta laws of the C-vine. One Gaussian vector and one chi-squared variate per row therefore carry all the Gamma-distributed information that either classical construction consumes: the onion regenerates one discarded Gamma component per row; under the tail representation the vine regenerates one fresh Gamma tail per partial correlation, while under the conventional symmetric-Beta representation it instead draws two Gamma components per parameter. The excess depends on the vine representation. Under the tail representation of Theorem~\ref{thm:vine_coupling} the nested tails of the theorem are literally regenerated, at an aggregate excess of $n(n-1)/2-(n-1)=(n-1)(n-2)/2$ fresh Gamma tails; the conventional symmetric-Beta implementation does not regenerate those nested tails, but draws two Gammas per parameter, at an aggregate excess of $n(n-1)-(n-1)=(n-1)^2$ Gamma-equivalent calls. Proposition~\ref{prop:rng_ledger} prices the redundancies, $n-1$ against $2(n-1)$ against $n(n-1)$ Gamma-equivalent draws under gamma-ratio accounting, and Proposition~\ref{prop:cholesky_density} derives the (known) Cholesky-factor density directly from the Bartlett rows, exposing the independent row structure on which both couplings rest.

Section~\ref{sec:characterizations} adds a variational reading of the family: relative to flat Lebesgue measure in the $n(n-1)/2$ free off-diagonal coordinates, $\operatorname{LKJ}_n(\eta)$ is the maximum-entropy distribution at fixed $\mathbb{E}[\log\det C]$, so the natural parameter $\eta-1$ is the Lagrange multiplier dual to that moment.

The practical consequences are three. The controlled benchmarks confirm the distinction the ledger predicts, between an onion advantage that vanishes in $n$ ($1.7\times$ at $n=3$) and a vine advantage that persists as a constant factor, roughly five to six for the conventional symmetric-beta implementation and about $2.5$--$2.7$ in the tail-representation experiment. The measured constants differ from the limiting ratios of the RNG-only cost model because the arithmetic omitted by that model is itself $O(n^2)$: the tail-representation plateau lies below its model benchmark, whereas the conventional-vine ratios fluctuate around theirs. Package-level speedups against \texttt{Distributions.jl} are larger, roughly $1.5\times$--$3.3\times$ for allocating calls and up to $4.5\times$ for preallocated calls, reflecting implementation overhead in addition to the saved draws. Finally, the direct Bartlett implementation is the most robust among the three implementations examined at small $\eta$: it contains no subtraction, and its diagonal entries vanish only near the subnormal threshold, against machine-epsilon thresholds for the textbook computations of $(1-R^2)^{1/2}$. The sampler is valid for all real $\eta>0$, requires only standard normal and chi-squared variates, and supports an allocation-free in-place interface (Algorithm~\ref{alg:bartlett_inplace}).

\ifincludetests
\appendix

\section{Implementation Tests}
\label{app:tests}

The distributional identities among the three samplers are proved exactly (\xthm{thm:coupling}{1} and \xthm{thm:vine_coupling}{2}), so the checks reported here test the implementation, not the theorems; they are regression tests, and agreement is the expected outcome. Large Kolmogorov--Smirnov $p$-values in particular are not evidence of equality and a correct implementation will occasionally produce small ones.

\paragraph{Pathwise coupling test.}
The most stringent implementation test follows directly from part~(iii) of \xthm{thm:coupling}{1}: feed the same realized $(Z_i,G_i)$ to both formulations, the assignments in \xalg{alg:bartlett}{1} and the onion assignments with $(X_i,Y_i)=(Q_i,G_i)$, and verify that every entry of the factor agrees to rounding error. Over $1{,}000$ replications of all rows at each of $(n,\eta)\in\{(5,1),(50,1),(50,0.7),(200,10)\}$, the maximum absolute discrepancy is $7.3\times10^{-14}$. At $(n,\eta)=(5,0.1)$ it rises to $9.0\times10^{-9}$, and the increase is itself informative: it is entirely the cancellation error of the onion-side assignment $(1-R_i^2)^{1/2}$, whose absolute error is of order $\varepsilon/L_{nn}$ when $G_n$ is tiny (\xsec{sec:ledger}{7}), while the Bartlett-side assignment $G_n^{1/2}/S_n$ remains accurate to a few ulps on the same draws. The coupling test thus doubles as a direct witness of the stability analysis. The analogous test for \xthm{thm:vine_coupling}{2} feeds the same realized $(Z_i,G_i)$ to the vine construction in \xeq{eq:vine_cholesky}{1} with $\rho_{ji;1\cdots j-1}=P_{ij}$ and to \xalg{alg:bartlett}{1} directly. The maximum discrepancy over the same benign configurations is $2.7\times10^{-15}$, rising to $8.8\times10^{-9}$ at $(n,\eta)=(5,0.1)$; the increase again reflects the $(1-\rho^2)^{1/2}$ cancellation, which the vine construction shares with the onion.

\paragraph{Row-level radius tests.}
For early, middle, and final rows ($i=2,\lceil n/2\rceil,n$) we test the empirical row law against its exact Beta distribution by one-sample Kolmogorov--Smirnov tests, including noninteger and small shapes ($\eta\in\{0.1,0.7\}$). The primary test is on the squared diagonal, $L_{ii}^2=1-R_i^2\sim\operatorname{Beta}(\eta+(n-i)/2,(i-1)/2)$; all twelve configurations pass, with $p$ between $0.08$ and $0.93$. The complementary test on $R_i^2$ itself agrees everywhere except the final row at $(n,\eta)=(10,0.1)$, where it returns $p\approx0$ while the $L_{ii}^2$ test of the same draws gives $p=0.13$. The disagreement is double-precision saturation, not a sampler discrepancy: in that configuration $222$ of $20{,}000$ draws ($1.1\%$) have $1-R_i^2$ small enough that the computed $R_i^2$ rounds to exactly one, so the empirical distribution acquires an atom at $1$ that the continuous reference law does not have, whereas $L_{ii}^2$ represents the same information down to the subnormal range. The asymmetry between the two parametrizations of one quantity is the same representational advantage documented in the stress tests below.

\paragraph{Moment checks with Monte Carlo uncertainty.}
Under $\operatorname{LKJ}_n(\eta)$, $\operatorname{var}(C_{ij})=1/(n+2\eta-1)$ for $i\neq j$ \citep[Proposition~1]{LewandowskiKurowickaJoe:2009}, and from the Wishart representation with $\nu=n+2\eta-1$,
$$
\mathbb{E}[\log\det C]=\sum_{i=1}^n\psi\left(\frac{\nu-i+1}{2}\right)-n\psi\left(\frac{\nu}{2}\right),
$$
where $\psi$ is the digamma function; see \citet{Hansen:CorrelationMatrix} for the large-$n$ expansion and the associated central limit theorem in the uniform case. Table~\ref{tab:validation} compares empirical values from both samplers with these closed forms, based on $N=50{,}000$ draws per configuration, with Monte Carlo standard errors. All empirical values lie within $3$ standard errors of theory, and the two-sample Kolmogorov--Smirnov $p$-values on the $C_{12}$ marginal are unremarkable, lying between $0.25$ and $0.97$.

\paragraph{Vine implementation tests.}
Three vine code paths appear in this paper and each is tested separately. The Bartlett-coupled shared-tail path is covered by the pathwise coupling test above, which verifies entrywise agreement with \xalg{alg:bartlett}{1} on shared inputs; being pathwise, that test subsumes any distributional check. The two independent-input paths require their own regression tests, since neither is a deterministic function of a Bartlett row.

The conventional symmetric-Beta gamma-ratio vine sampler, the one benchmarked in \xtab{tab:controlled_vine}{4} and stress-tested in Table~\ref{tab:stress} below, is validated in the same way as the Bartlett and onion samplers. Across the configurations $(n,\eta)\in\{(5,1),(10,2),(20,0.5),(10,0.7)\}$ with $N=10{,}000$ draws, all empirical values of $\operatorname{var}(C_{12})$ and $\mathbb{E}[\log\det C]$ lie within $1.5$ Monte Carlo standard errors of their theoretical values, and two-sample Kolmogorov--Smirnov tests of the $C_{12}$ marginal give $p$-values between $0.12$ and $0.90$ against both the Bartlett and the \texttt{Distributions.jl} implementations. The fresh-tail sampler of \xtab{tab:controlled_vine_thm}{3}, which realizes the tail representation of \xthm{thm:vine_coupling}{2} with independently drawn tails, is validated by the same battery over the same configurations; the corresponding output is in the repository.

\paragraph{Stress tests at extreme parameters.}
For $\eta\in\{10^{-6},10^{-3},0.1\}$, Table~\ref{tab:stress} records the frequency of an exactly zero diagonal entry for all three samplers, together with the predictions of \xsec{sec:ledger}{7} evaluated at the ideal rounding thresholds: $\Pr(G_n<t)\sim(t/2)^{\eta}/\Gamma(\eta+1)$ with $t\approx2.5\times10^{-324}$ for Bartlett; $\Pr(1-R_n^2<\delta)\sim\delta^{\eta}/\{\eta B(\eta,(n-1)/2)\}$ with $\delta=2^{-54}\approx5.6\times10^{-17}$ for the onion, this being the point below which an exact $1-\delta$ rounds to unity rather than the spacing $2^{-53}$ below one; and the corresponding two-sided Beta tail for the conventional vine implementation, whose final tree has the $n$-free law $\operatorname{Beta}(\eta,\eta)$. These are leading-order tail approximations with exact constants, not exact probabilities. Observed and predicted frequencies agree to about two digits throughout, including the $n$-dependence of the onion rate, which enters through the constant $\eta B(\eta,(n-1)/2)$. The vine prediction is the loosest of the three: its partial correlation passes through the affine map $2V-1$, a squaring, and a product accumulation, so its effective zero threshold is not the single subtraction at one that governs the onion. The $\Pr(\mathrm{NaN})$ column likewise describes the conventional gamma-ratio vine implementation benchmarked here, in which both $\operatorname{Gamma}(\eta,2)$ components of a symmetric-Beta draw can underflow and yield $0/0$. It is not a property of the C-vine construction itself, nor of Stan Math, whose small-shape Beta branch works in log space and so avoids that particular failure.

The conventional vine implementation has a second failure mode that the other two samplers do not: both Gamma components of a ratio can underflow to zero, producing an invalid (NaN) partial correlation. Because a NaN draw carries no diagonal entry to test, the zero-diagonal frequencies are necessarily conditional on a valid draw, and Table~\ref{tab:stress} therefore reports $\Pr(\mathrm{NaN})$, the number of valid draws, the conditional zero-diagonal frequency, and the unconditional probability of either failure. The distinction matters for the vine. At $\eta=10^{-3}$ its conditional zero-diagonal frequency is $0.964$, indistinguishable from the onion's, but $22\%$ of draws are invalid, so the unconditional failure probability is $1-(1-0.225)(1-0.964)\approx0.972$; the predicted NaN rate is $0.476^2\approx0.23$. At $\eta=10^{-6}$ essentially every vine draw is invalid, leaving $144$ valid draws out of $10^5$ at $n=5$ and $12$ out of $10^4$ at $n=50$, so the conditional frequencies in that row carry almost no information and are reported only for completeness.

Setting the vine's NaNs aside, the zero-diagonal thresholds differ by many orders of magnitude: at $\eta=10^{-3}$ the Bartlett sampler produces a zero diagonal in $47\%$ of draws against $96$--$97\%$ for the onion, and at $\eta=0.1$ Bartlett produced none in $10^5$ draws (minimum diagonal $\approx10^{-28}$, predicted failure probability $\approx10^{-32}$) while the onion and conventional vine produced exact zeros in about $3\%$ and $2.3\%$ of draws respectively. At $\eta=10^{-6}$ all samplers fail in essentially every draw; that regime is not reliably representable in ordinary double-precision Cholesky coordinates. Row-norm errors $\max_i|\sum_{j\leq i}L_{ij}^2-1|$ never exceeded $2.0\times10^{-15}$ for any sampler, consistent with the few-ulps claim of \xsec{sec:ledger}{7}.

We define numerical non-positive-definiteness of the \emph{formed product} $LL^\prime$ as failure of its double-precision Cholesky factorization, that is, a nonpositive computed pivot. At $\eta=0.1$ the formed products fail for about $2\%$ of draws at $n=5$ and about $3\%$ at $n=50$, at nearly identical rates for the three samplers, which locates that failure in the formation of $C$ rather than in either factor, a further argument for working with the factor directly.
Every negative computed eigenvalue of $LL^\prime$, across all configurations and samplers, exceeded $-1.5\times10^{-15}$, and the frequencies of $\lambda_{\min}<-10^{-14}n$ and $\lambda_{\min}<-10^{-12}n$ were exactly zero. The failures are therefore roundoff-scale artifacts of forming the product, not material indefiniteness of either factor.

\begin{table}[t]
\centering
\caption{Stress tests at extreme parameters. $N=10^5$ draws for $n=5$, $N=10^4$ for $n=50$. $\Pr(\mathrm{NaN})$ is the unconditional frequency of an invalid factor; $N_{\mathrm{val}}$ is the resulting number of valid draws; the observed and predicted zero-diagonal frequencies are conditional on a valid draw; and $\Pr(\text{fail})=1-\{1-\Pr(\mathrm{NaN})\}\{1-\Pr(0\mid\text{val})\}$ is the unconditional frequency of either failure mode. Predictions are the leading-order tail approximations with exact constants given in \xsec{sec:ledger}{7}, evaluated at the ideal rounding thresholds. Realized thresholds can also depend on generator internals.}
\label{tab:stress}
\small
\begin{tabular}{rr l rrrrrr}
\toprule
$\eta$ & $n$ & Sampler & $\Pr(\mathrm{NaN})$ & $N_{\mathrm{val}}$ & \makecell{Observed\\$\Pr(0\mid\text{val})$} & \makecell{Predicted\\$\Pr(0\mid\text{val})$} & $\Pr(\text{fail})$ & Min.\ diagonal \\
\midrule
$10^{-6}$ &  5 & Bartlett & 0     & 100{,}000 & 0.999 & 0.999      & 0.999 & 0 \\
          &    & Onion    & 0     & 100{,}000 & 1.000 & 1.000      & 1.000 & 0 \\
          &    & Vine     & 0.999 & 144       & 1.000 & 1.000      & 1.000 & 0 \\
$10^{-6}$ & 50 & Bartlett & 0     & 10{,}000  & 0.999 & 0.999      & 0.999 & 0 \\
          &    & Onion    & 0     & 10{,}000  & 1.000 & 1.000      & 1.000 & 0 \\
          &    & Vine     & 0.999 & 12        & 1.000 & 1.000      & 1.000 & 0 \\
$10^{-3}$ &  5 & Bartlett & 0     & 100{,}000 & 0.476 & 0.475      & 0.476 & 0 \\
          &    & Onion    & 0     & 100{,}000 & 0.965 & 0.964      & 0.965 & 0 \\
          &    & Vine     & 0.225 & 77{,}521  & 0.964 & 0.963      & 0.972 & 0 \\
$10^{-3}$ & 50 & Bartlett & 0     & 10{,}000  & 0.470 & 0.475      & 0.470 & 0 \\
          &    & Onion    & 0     & 10{,}000  & 0.966 & 0.967      & 0.966 & 0 \\
          &    & Vine     & 0.226 & 7{,}744   & 0.963 & 0.963      & 0.971 & 0 \\
$0.1$     &  5 & Bartlett & 0     & 100{,}000 & 0     & $10^{-32}$ & 0     & $2.4\times10^{-28}$ \\
          &    & Onion    & 0     & 100{,}000 & 0.027 & 0.026      & 0.027 & 0 \\
          &    & Vine     & 0     & 100{,}000 & 0.023 & 0.024      & 0.023 & 0 \\
$0.1$     & 50 & Bartlett & 0     & 10{,}000  & 0     & $10^{-32}$ & 0     & $5.0\times10^{-23}$ \\
          &    & Onion    & 0     & 10{,}000  & 0.035 & 0.034      & 0.035 & 0 \\
          &    & Vine     & 0     & 10{,}000  & 0.029 & 0.024      & 0.029 & 0 \\
\bottomrule
\end{tabular}
\end{table}

\begin{table}[t]
\centering
\caption{Regression tests of the implementations ($N=50{,}000$ draws per configuration).
Var denotes the variance of $C_{12}$ and ELD the mean of $\log\det C$;
``th'' is the theoretical value, ``Ba'' and ``On'' are empirical values from the Bartlett and onion samplers, with Monte Carlo standard errors in parentheses.
KS $p$ is the $p$-value from a two-sample Kolmogorov--Smirnov test on the $C_{12}$ marginal; since the distributional identity is proved, these are checks on the implementations only.}
\label{tab:validation}
\small
\begin{tabular}{rr rrr rrr r}
\toprule
$n$ & $\eta$
  & Var (th) & Var (Ba) & Var (On)
  & ELD (th) & ELD (Ba) & ELD (On)
  & KS $p$ \\
\midrule
  5 & 1.0 & 0.1667 & 0.1677 (8)  & 0.1653 (8)  & $-3.106$ & $-3.099$ (7)  & $-3.094$ (7)  & 0.713 \\
 10 & 2.0 & 0.0769 & 0.0773 (4)  & 0.0770 (4)  & $-5.516$ & $-5.519$ (5)  & $-5.500$ (5)  & 0.970 \\
 10 & 0.7 & 0.0962 & 0.0972 (5)  & 0.0966 (5)  & $-8.990$ & $-8.985$ (10) & $-8.992$ (10) & 0.744 \\
 20 & 0.5 & 0.0500 & 0.0506 (3)  & 0.0503 (3)  & $-20.621$ & $-20.613$ (14) & $-20.619$ (14) & 0.508 \\
 50 & 1.0 & 0.0196 & 0.0196 (1)  & 0.0196 (1)  & $-46.914$ & $-46.923$ (11) & $-46.915$ (11) & 0.254 \\
\bottomrule
\end{tabular}
\par\smallskip
{\footnotesize Standard errors in units of the last displayed digit for Var, and of $10^{-3}$ for ELD.}
\end{table}

\fi

\FloatBarrier
\clearpage

\section*{Declarations}

\paragraph{Funding.}
This research did not receive any specific grant from funding agencies in the public, commercial, or not-for-profit sectors.

\paragraph{Declaration of competing interest.}
The author declares that he has no known competing financial interests or personal relationships that could have appeared to influence the work reported in this paper.

\paragraph{Declaration of generative AI and AI-assisted technologies in the manuscript preparation process.}
Generative AI tools, including Claude (Anthropic) and ChatGPT (OpenAI), were used to review and copy-edit the manuscript, assist with implementing the simulation designs, reorganize the presentation of results, and prepare the supplementary material. The author verified all mathematical statements, proofs, numerical results, and conclusions and takes full responsibility for the final content.

\paragraph{Data and code availability.}
This paper reports no empirical data. The Julia code that generates every table and test reported in the paper and in \testslocation, together with the environment files that pin exact package versions, is archived and citable; see the Code availability paragraph in Section~\ref{sec:numerical}.

\section*{Acknowledgements}

I thank S{\o}ren Johansen for drawing my attention to Barndorff-Nielsen and Schou (1973).

\end{document}